\documentclass[12pt]{amsart}

\usepackage{amsthm, mathrsfs, amsfonts,amssymb}
\usepackage{enumerate}

\theoremstyle{plain}
\newtheorem{thm}{\sc Theorem}[section]
\newtheorem{defn}[thm]{\sc Definition}
\newtheorem{lem}[thm]{\sc Lemma}
\newtheorem{prop}[thm]{\sc Proposition}
\newtheorem{cor}[thm]{\sc Corollary}
\newtheorem{rem}[thm]{\sc Remark}

\title[Almost Periodic and Uniformly Continuous functionals]{Weakly Almost Periodic and Uniformly Continuous Functionals on the Orlicz Fig\`{a}-Talamanca Herz algebras}

\author{Rattan Lal}
\address{Rattan Lal,\newline\indent Department of Mathematics,\newline\indent Indian Institute of Technology Delhi,\newline\indent Delhi - 110016, India.}
\email{rattanlaltank@gmail.com}

\author{N. Shravan Kumar}
\address{N. Shravan Kumar,\newline\indent Department of Mathematics,\newline\indent Indian Institute of Technology Delhi,\newline\indent Delhi - 110016, India.}
\email{shravankumar@maths.iitd.ac.in}

\begin{document}

\begin{abstract}
In this paper we study weakly almost periodic and uniformly continuous functionals on the Orlicz Fig\`{a}-Talamanca Herz algebras associated to a locally compact group. We show that a unique invariant mean exists on the space of weakly almost periodic functionals. We also characterise discrete groups in terms of the inclusion of the space of uniformly continuous functions inside the space of weakly almost periodic functionals.
\end{abstract}


\maketitle

\section{Introduction}
Let $G$ be a locally compact group and let $A(G)$ denote the Fourier algebra associated with $G.$ In 1973, Dunkl and Ramirez \cite{DR} defined the concept of almost periodic functionals on the Fourier algebra, extending the classical notion of almost periodic functions on locally compact abelian groups. Let $1<p<\infty$ and let $A_p(G)$ denote the Fig\`{a}-Talamanca Herz algebra introduced and studied by Herz \cite{H}. In \cite{G1}, Granirer extended the results of Dunkl and Ramirez to the $A_p(G)$ algebras. 

In \cite{LK1}, the authors have introduced and studied the $L^\Phi$-versions of the Fig\`{a}-Talamanca Herz algebras. Here $L^\Phi$ denotes the Orlicz space corresponding to the Young function $\Phi.$ The space $A_\Phi(G)$ is defined as the space of all continuous functions $u,$ where $u$ is of the form $$u=\underset{n=1}{\overset{\infty}{\sum}}f_n*\check{g_n},$$ where $f_n\in L^\Phi(G),$ $g_n\in L^\Psi(G),$ $(\Phi,\Psi)$ is a pair of complementary Young functions satisfying the $\Delta_2$-condition and $$\underset{n=1}{\overset{\infty}{\sum}}N_\Phi(f_n)\|g_n\|_\psi<\infty.$$ 

The aim of this paper is to extend the results of Granirer to the context of $A_\Phi(G)$ algebras. More precisely, in section 3, we define the notion of weakly almost periodic functions and study its properties. We also show that a unique invariant mean exists on the space of weakly almost periodic functionals. In section 4, we take the study of uniformly continuous functionals and show that it is an algebra.

\section{Preliminaries}

A convex function $ \Phi: \mathbb{R}\rightarrow [0,\infty] $ is called a Young function if it is symmetric and  satisfies $ \Phi(0)= 0 $ and $\underset{x\rightarrow \infty}{\lim}  \Phi(x)= + \infty $. For any Young function $ \Phi,$ define another Young function $\Psi$ as
$$ \Psi(y):= \sup{\{x|y|-\Phi(x) : x\geq0\}} , y\in\mathbb{R}.$$   This Young function $\Psi $ is called as the complementary function to $ \Phi $ and 
the pair $ (\Phi,\Psi) $ is called a complementary pair of Young functions. 

Let $G$ be a locally compact group with a left Haar measure $dx.$ We say that a  Young function $ \Phi $  satisfies  the $\Delta_{2}$-condition, denoted $\Phi\in\Delta_{2},$ if there exists a constant $ K>0 $ and $ x_{0} > 0$ such that $\Phi(2x)\leq K\Phi(x)$ whenever $x\geq x_{0}$ if $G$ is compact and the same inequality holds with $ x_{0}=0 $ if $G$ is non compact.

The Orlicz space, denoted $ L^{\Phi}(G),$ is a vector space consisting of measurable functions, defined as $$ {L}^{\Phi}(G) = \left\{ f: G \rightarrow  \mathbb{C}:f \mbox{ is measurable and }\int_G\Phi(\beta |f|)\ dx <\infty \text{ for some}~ \beta>0  \right\} $$ The Orlicz space $ L^{\Phi}(G) $ is a Banach space when equipped with the norm
$$N_{\Phi}(f) = \inf \left\{ k>0 :\int_G\Phi\left(\frac{|f|}{k}\right) dx \leq1 \right\}.$$ The above norm is called as the Luxemburg norm or Gauge norm.  If $(\Phi,\Psi)$ is a complementary Young pair, then there is a norm on $L^\Phi(G),$ equivalent to the Luxemberg norm, given by, $$ \|f\|_{\Phi} =\sup \Bigg \{ \int_{G}|fg|dx :g\in L^\Psi(G), \int_{G}\Psi(|g|)dx\leq1 \Bigg\}.$$ This norm is called as the Orlicz norm. 

Let $C_{c}(G)$ denote the space of all continuous functions on $ G $ with compact support. If a Young function $ \Phi $  satisfies  the $\Delta_{2}$ -condition, then $C_c(G)$ is dense in $L^\Phi(G).$ Further, if the complementary function $\Psi$ is such that $\Psi$ is continuous and $\Psi(x)=0$ iff $x=0,$ then the dual of $ (L^{\Phi}(G),N_{\Phi}(\cdot)) $ is isometrically isomorphic to $ (L^{\Psi}(G),\|\cdot\|_{\Psi}).$ In particular, if both $\Phi$ and $\Psi$ satisfies the $\Delta_2$-condition, then $L^\Phi(G)$ is reflexive.

The following is the Orlicz space analogue of \cite[Proposition 2.42]{F}.

\begin{thm}\label{app_id}
Let $\Phi$ be a $\Delta_2$-regular Young function and let $\{f_\alpha\}$ be a bounded approximate identity for $L^1(G).$ Then, for any $g\in L^\Phi(G),$ we have $$\underset{\alpha}{\lim}\ \|g*f_\alpha-g\|_\Phi=0.$$
\end{thm}

\noindent For more details on Orlicz spaces, we refer the readers to \cite{RR}.

Let $\Phi$ and $\Psi$ be a pair of complementary Young functions satisfying the $\Delta_2$ condition. Let $$A_\Phi(G)=\left\{u=\underset{n=1}{\overset{\infty}{\sum}}f_n*\check{g_n}:\{f_n\}\subset L^\Phi(G),\{g_n\}\subset L^\Psi(G)\mbox{ and }\underset{n=1}{\overset{\infty}{\sum}}N_\Phi(f_n)\|g_n\|_\Psi<\infty\right\}.$$ Note that if $u\in A_\Phi(G)$ then $u\in C_0(G).$ If $u\in A_\Phi(G),$ define $\|u\|_{A_\Phi}$ as $$\|u\|_{A_\Phi}:=\inf\left\{\underset{n=1}{\overset{\infty}{\sum}}N_\Phi(f_n)\|g_n\|_\Psi:u=\underset{n=1}{\overset{\infty}{\sum}}f_n*\check{g_n}\right\}.$$ The space $A_\Phi(G)$ equipped with the above norm and with the pointwise addition and multiplication becomes a commutative Banach algebra \cite[Theorem 3.4]{LK1}. In fact, $A_\Phi(G)$ is a commutative, regular and semisimple Banach algebra with spectrum homeomorphic to $G$ \cite[Corollary 3.8]{LK1}. This Banach algebra $A_\Phi(G)$ is called as the Orlicz Fig\`{a}-Talamanca Herz algebra.

Let $ \mathcal{B}(L^{\Phi}(G))$ be the linear space of all bounded linear operators on $L^{\Phi}(G)$ equipped with the operator norm. For a bounded complex Radon measure $\mu$  on $G$ and $f\in L^{\Phi}(G),$ define $T_\mu:L^\Phi(G)\rightarrow L^\Phi(G)$ by $ T_{\mu}(f)=\mu*f.$ It is clear that $T_\mu\in \mathcal{B}(L^{\Phi}(G)).$ Let $ PM_{\Phi}(G)$ denote the closure of $\{T_{\mu}:\mu\mbox{ is a bounded complex Radon measure}\}$ in $\mathcal{B}(L^{\Phi}(G))$ with respect to the ultraweak topology. It is proved in \cite[Theorem 3.5]{LK1}, that for a locally compact group $G,$ the dual of $ A_{\Phi}(G)$ is isometrically isomorphic to $PM_{\Psi}(G).$ 

For further details see \cite{LK1}.

We shall denote by $B_\Phi(G)$ the space of all continuous functions $v:G\rightarrow \mathbb{C}$ such that $uv\in A_\phi(G)$ for all $u\in A_\Phi(G).$ Equip $B_\Phi(G)$ with the operator norm. With this norm $B_\Phi(G)$ becomes a commutative Banach algebra.

For $ u \in  B_{\Phi}(G) $ and $ T \in  PM_{\Psi}(G),$ define $ u.T \in  PM_{\Psi}(G) $ by $$ \langle v ,u.T \rangle :=  \langle u v ,T \rangle ~~~ \forall ~~v \in A_{\Phi}(G). $$ With this action, $PM_{\Psi}(G) $ becomes a $  B_{\Phi}(G) $-module. Further, if $ m \in  PM_{\Psi}(G)^{\prime} $ and $ u \in  B_{\Phi}(G),$ define $ u.m \in  PM_{\Psi}(G)^{\prime} $ by $$ \langle T ,u.m \rangle :=  \langle u.T ,m \rangle ~~~ \forall ~~T \in PM_{\Psi}(G). $$ This action makes $PM_{\Psi}(G)^{\prime} $  also into a $  B_{\Phi}(G) $-module. We now define a multiplication on $PM_{\Psi}(G)^{\prime} $, called Arens multiplication, as follows. For $ m \in  PM_{\Psi}(G)^{\prime} $ and $ T \in  PM_{\Psi}(G) $, define $ m \odot T \in PM_{\Psi}(G) $ as $$ \langle v ,m \odot T \rangle :=  \langle v.T ,m \rangle ~~~ \forall ~~v \in A_{\Phi}(G). $$ 

Similarly, if $u\in A_\Phi(G)$ and $f\otimes g\in L^\Phi(G)\widehat{\otimes}_{L^1(G)}L^\Psi(G),$ define $u.(f\otimes g)(x,y)=u(yx^{-1})f(x)g(y),$ $x,y\in G.$ Here $L^\Phi(G)\widehat{\otimes}_{L^1(G)}L^\Psi(G)$ denotes the module tensor product. For more details see \cite{Rie}. Note that $u.(f\otimes g)\in L^\Phi(G)\widehat{\otimes}_{L^1(G)}L^\Psi(G)$ and thus with the above action, the space $L^\Phi(G)\widehat{\otimes}_{L^1(G)}L^\Psi(G)$ becomes an $A_\Phi(G)$-module. Further, the map $f\otimes g\mapsto f*\check{g}$ from $L^\Phi(G)\widehat{\otimes}_{L^1(G)}L^\Psi(G)$ to $A_\Phi(G)$ is an $A_\Phi(G)$-module map. 

We shall denote by $CV_\Psi(G)$ the space of all bounded linear operators $T:L^\Psi(G)\rightarrow L^\Psi(G)$ such that $T(f*g)=T(f)*g$ for all $f\in L^\Psi(G)$ and $g\in L^1(G).$ It now follows from \cite[Corollary 2.13]{Rie} that the topological dual of $L^\Phi(G)\widehat{\otimes}_{L^1(G)}L^\Psi(G)$ is isometrically isomorphic to $CV_\Psi(G).$ It is clear that $PM_\Psi(G)$ is contained in $CV_\Psi(G)$ and in fact this inclusion is an $A_\Phi(G)$-module map.

Throughout this paper, $G$ will denote a locally compact group and $(\Phi,\Psi)$ will denote a pair of complementary Young functions satisfying the $\Delta_2$-condition.

\section{Weakly almost periodic functionals}
In this section, we study the properties of the space of weakly almost periodic functionals of $A_\Phi(G).$ Our main aim in this section is to show that a unique invariant mean exists on the space of weakly almost periodic functionals.
\begin{defn}
An almost periodic functional on $A_\Phi(G)$ is a functional $T\in PM_\Psi(G)$ such that the mapping $u\mapsto u.T$ from $A_\Phi(G)$ into $PM_\Psi(G)$ is a compact operator. Similarly, a weakly almost periodic functional on $A_\Phi(G)$ is  a functional $T\in PM_\Psi(G)$ for which the mapping $u\mapsto u.T$ from $A_\Phi(G)$ into $PM_\Psi(G)$ is a weakly compact operator. 
\end{defn}
We denote by $AP_\Psi(\widehat{G})$ and $WAP_\Psi(\widehat{G})$ the space of all almost periodic functionals and weakly almost periodic functionals respectively.
\begin{lem}\label{WAP_Prop}
\mbox{}
\begin{enumerate}[(i)]
\item The spaces $AP_\Psi(\widehat{G})$ and $WAP_\Psi(\widehat{G})$ are norm closed $B_\Phi(G)$- submodules of $PM_\Psi(G).$
\item The identity operator $I\in AP_\Psi(\widehat{G}).$
\item The inclusion $AP_\Psi(\widehat{G})\subset WAP_\Psi(\widehat{G})$ holds.
\end{enumerate}
\end{lem}
\begin{proof}
Note that (i) and (iii) are consequences of the definition of norm and the compact operator. Further, for any $u,v\in A_\Phi(G),$ $$\langle v,u.I\rangle=\langle vu,I\rangle=(vu)(e)=u(e)\langle v,I\rangle=\langle v,u(e)I\rangle.$$ Since $v\in A_\Phi(G)$ is arbitrary, it follows that $u.I=u(e)I$ and thus (ii) is proved.
\end{proof}
Let $M(G)$ denote the space of all complex valued bounded Radon measures on $G.$ If $\mu\in M(G),$ then $\mu$ can be considered as a linear functional on $A_\Phi(G)$ as $$<\mu,\varphi>:=\int \varphi d\mu$$ for $\varphi\in A_\Phi(G).$ By definition, it is also clear that $\mu$ is continuous and $$\|\mu\|_{PM_\Psi(G)}\leq\|\mu\|_{M(G)}.$$
\begin{prop}
\mbox{}
\begin{enumerate}[(i)]
\item The space $M(G),$ considered as a subspace of $PM_\Psi(G),$ is a $B_\Phi(G)$-submodule.
\item The closure $\overline{M(G)}^{\|\cdot\|_{PM_\Psi(G)}}$ is contained in $WAP_\Psi(\widehat{G}).$
\item The space $\ell^1(G)$ is contained in $AP_\Psi(\widehat{G}).$
\end{enumerate}
\end{prop}
\begin{proof}
(i) Let $\mu\in M(G)$ and $u\in B_\Phi(G).$ Then for $v\in A_\Phi(G),$ we have $$<v,u.\mu>=<uv,\mu>=\int uv\ d\mu= \int v(u d\mu).$$ Thus (i) follows. \\
(ii) It is enough to prove that $M(G)\subset WAP_\Psi(\widehat{G}).$ In order to prove this, it is enough to show that the mapping $u\mapsto u.\mu$ from $A_\Phi(G)$ to $PM_\Psi(G)$ is a weakly compact operator, for a probability measure $\mu$ on $G.$ Note that this is equivalent to saying that the set $\{u.\mu:u\in A_\Phi(G),\|u\|_{A_\Phi(G)}\leq1\}$ is relatively weakly sequentially compact, for a probability measure $\mu$ on $G.$ 

Let $\mu\in M(G)$ be a probability measure. Define $S:L^\Phi(G,d\mu)\rightarrow PM_\Psi(G)$ as $$\langle v,Sf\rangle=\int_Gvf\ d\mu,\ f\in L^\Phi(G,d\mu)\mbox{ and }v\in A_\Phi(G).$$ Since $\mu$ is a probability measure, it follows that $$A_\Phi(G)\subset L^\infty(G,d\mu)\subset L^\Phi(G,d\mu)\subset L^1(G,d\mu)$$ and hence $S$ is continuous, in particular weakly continuous. Further, it follows from the definition of the operator $S$ and from (i) that, for $u\in A_\Phi(G),$ $Su=u.\mu.$ The conclusion now follows from the fact that the unit ball of $L^\Phi(G)$ is weakly sequentially compact.

(iii) As any element of $\ell^1(G)$ is a $\|\cdot \|_{M(G)}$ limit of finite linear combinations of point masses, it is enough to prove that the point masses are in $AP_\Psi(\widehat{G}).$ For any $x\in G,$ $\{u\delta_x:u\in A_\Phi(G),\|u\|_{A_\Phi(G)}\leq 1\}\subset\{\alpha\delta_x:\alpha\in\mathbb{C},|\alpha|\leq1\}$ and hence (iii) follows.
\end{proof}
Here is the promised result on the existence of a unique invariant mean on $WAP_\Psi(\widehat{G}).$
\begin{thm}\label{WAP_Inv_Mean}
The space $WAP_\Psi(\widehat{G})$ has a unique topological invariant mean $m$ such that, for $\mu\in M(G),$ $m(\mu)=\mu(\{e\}),$ where $e$ denotes the identity element of the group $G.$
\end{thm}
\begin{proof} 
(i) By \cite[Corollary 6.2]{LK1}, $TIM_\Psi(\widehat{G})$ is non-empty. Let $m\in TIM_\Psi(\widehat{G}).$ Now, $m$ restricted to $WAP_\Psi(\widehat{G})$ is a topological invariant mean on $WAP_\Psi(\widehat{G}).$ This proves the existence of an invariant mean on $WAP_\Psi(\widehat{G}).$

(ii) Let $T\in WAP_\Psi(\widehat{G}).$ Let $m$ be the invariant mean provided by (i). By Goldstine's theorem, there exists a net $\{u_\alpha\}\subset A_\Phi(G)$ such that $\|u_\alpha\|_{A_\Phi}\leq 1$ and $u_\alpha\rightarrow m$ in the weak*-topology. In particular, $u_\alpha(e)\rightarrow 1.$ For any $u\in A_\Phi(G),$ we have 
\begin{align*}
<u,u_\alpha.T>=&<uu_\alpha, T> = <u_\alpha,u.T>\\ \rightarrow& <u.T, m> = u(e)<T,m> \\ =& <T,m> I(u).
\end{align*}
Thus, $u_\alpha.T\rightarrow <T,m>I$ in the weak*-topology. On the other hand, as $T\in WAP_\Psi(\widehat{G}),$ it follows that there exists a subnet $\{u_{\alpha_\beta}\}$ of $\{u_\alpha\}$ such that $u_{\alpha_\beta}.T$ converges in the weak topology of $PM_\Psi(G).$ Since weak convergence implies weak* convergence, it follows that $$u_{\alpha_\beta}.T\rightarrow <T,m>I$$ in the weak topology.

(iii) We now prove the uniqueness. Let $m^\prime$ be any topological invariant mean on $WAP_\Psi(\widehat{G})$ and let $m$ be the invariant mean provided by (i). Let $T\in WAP_\Psi(\widehat{G}).$ By (ii), there exists a net $\{u_\alpha\}$ in the unit ball of $A_\Phi(G)$ such that $u_\alpha\rightarrow m$ in the weak*-topology. Further, there exists a subnet $\{u_{\alpha_\beta}\}$ of $\{u_\alpha\}$ such that $u_{\alpha_\beta}.T\rightarrow <T,m>I$ in the weak-topology. In particular, $$<u_{\alpha_\beta}.T,m^\prime>\rightarrow<<T,m>I,m^\prime>=<T,m>.$$ Also, it follows from (ii) that, $$<u_{\alpha_\beta}.T,m^\prime>=u_{\alpha_\beta}(e)<T,m^\prime>\rightarrow<T,m^\prime>.$$ Thus $m^\prime$ coincides with $m$ restricted to $WAP_\Psi(\widehat{G}).$ This proves the uniqueness of the invariant mean. 

(iv) Let $m$ be the unique topological invariant mean provided by (iii). We now claim that $<\mu,m>=\mu(\{e\})$ for all $\mu\in M(G).$ Note that, it is enough to prove for positive $\mu\in M(G).$ Let $0\leq \mu\in M(G).$ Let $\{U_n\}_{n\in\mathbb{N}}$ be a sequence of neighbourhoods of $e$ in $G$ such that $\overline{U_n}$ is compact, $U_{n+1}^2\subset U_n$ and $\mu(U_n)\rightarrow\mu(\{e\}).$ For each $n\in\mathbb{N},$ let $u_n$ be the function provided by \cite[Proposition 5.5]{LK1}, corresponding to the neighbourhood $U_n.$ Note that $u_n\rightarrow \chi_{U},$ where $V=\underset{n\in\mathbb{N}}{\cap}U_n.$ As $\chi_U=\chi_{\{e\}},\ \mu\ a.e.,$ it follows that $u_n\rightarrow \chi_{\{e\}},\mu\ a.e..$ Thus, for any $u\in A_\Phi(G),$ we have, $u_n.u\rightarrow u(e)\chi_{\{e\}},\mu\ a.e.$ and hence, by dominated convergence theorem, 
\begin{align*}
<u,u_n.\mu> = <u_nu,\mu> = \int_G u_nu\ d\mu \rightarrow \int_G u(e)\chi_{\{e\}}\ d\mu = <u,\mu(\{e\})\delta_e>,
\end{align*} 
i.e., $u_n.\mu\rightarrow \mu(\{e\})\delta_e$ in the weak*-topology of $PM_\Psi(G).$ As $u_n$'s are bounded, and $\mu\in WAP_\Psi(\widehat{G}),$ there exists a subsequence $\{u_{n_k}\}$ of $\{u_n\}$ such that $u_{n_k}.\mu$ converges weakly and hence in the weak*-topology. Thus, $u_{n_k}.\mu\rightarrow \mu(\{e\})\delta_e$ in the weak topology. As, $m$ is a topological invariant mean, it follows that $m(\mu)=m(u_n.\mu)$ for all $n.$ Thus, 
\begin{align*}
m(\mu)=&m(\mu(\{e\})\delta_e)=\mu(\{e\})m(\delta_e)=\mu(\{e\}). \qedhere
\end{align*}
\end{proof}

\section{Uniformly continuous functionals}

\begin{defn}
We define the following subspaces of $PM_\Psi(G):$
\begin{enumerate}
\item $\mathscr{M}(\widehat{G}):=\overline{M(G)};$
\item $\mathscr{M}^d(\widehat{G}):=\overline{\ell^1(G)};$
\item $PF_\Psi(G):=\overline{L^1(G)};$
\item $UCB_\Psi(\widehat{G}):=\overline{\{A_\Phi(G).PM_\Psi(G)\}};$
\item $C_\Psi(\widehat{G}):=\left\{T\in PM_\Psi(G):\mbox{if }S_1,S_2\in PF_\Psi(G),\mbox{ then }S_1T+TS_2\in PF_\Psi(G)\right\};$
\end{enumerate}
where bar denotes the closure in $\|\cdot\|_{PM_\Psi(G)}$-norm.
\end{defn}

\begin{prop}\label{PTC}
Let $\{u_n\}\subset B_\Phi(G).$ Then $u_n\rightarrow \chi_{\{e\}}$ in the pointwise topology if and only if $\langle u_n,T\rangle$ converges to $\langle T,m\rangle,$ for every $T\in \mathscr{M}^d(\widehat{G}).$
\end{prop}
\begin{proof}
First, let $x\neq e.$ Choose $u\in A_\Phi(G)$ such that $u(e)=0$ and $u(x)\neq 0.$ Note that
\begin{align*}
\langle \delta_x,m \rangle =& u(x)\langle \delta_x,m\rangle \\ =& \langle u(x)\delta_x,m\rangle\\ =& \langle u\cdot\delta_x,m\rangle\\ =& u(e)\langle \delta_x,m\rangle = 0. 
\end{align*}
Thus, $$\langle \delta_x,m \rangle=\left\{\begin{array}{ccc}
1 & \mbox{if} & x=e \\
0 & \mbox{if} & x\neq e
\end{array}\right..$$ Further, observe that, for every $x\in G,$ $\langle u_n,\delta_x \rangle = u_n(x).$ Hence the proof follows from these equalities.
\end{proof}

\begin{defn}
The support of an operator $T\in CV_\Psi(G),$ denoted $\mbox{supp}(T),$ is defined as the set of all $x\in G$ such that for every open set $U$ containing $e$ and for every open set $V$ containing $x$ there exists $f,g\in C_c(G)$ such that $\mbox{supp}(f)\subset U,$ $\mbox{supp}(g)\subset V$ and $\langle T(f),g \rangle\neq 0.$ 
\end{defn}

Since $A_\Phi(G)$ is a regular Tauberian algebra and also as the canonical inclusion of $PM_\Psi(G)$ inside $CV_\Psi(G)$ is an $A_\Phi(G)$-module map, we have the following corollary.
\begin{cor}\label{equal}
If $T\in PM_\Psi(G),$ then the support of $T$ as an element of the dual of $A_\Phi(G)$ coincides with the support of $T$ as an element of $CV_\Psi(G).$
\end{cor}

For $f\in L^1(G),$ define $T_f:L^\Phi(G)\rightarrow L^\Phi(G)$ by $T_f(g)=f*g.$ It is clear that $T_f\in CV_\Phi(G),$ for all $f\in L^1(G).$

An immediate consequence of Theorem \ref{app_id} is the following Lemma. For a proof of this see \cite[Proposition 9]{H}, where the case $\Phi(x)=|x|^p,$ $1<p<\infty,$ is considered.
\begin{lem}\label{App}
Let $T\in CV_\Phi(G)$ be such that $\mbox{supp}(T)$ is compact. Then for each open set $U$ containing $\mbox{supp}(T),$ there exists a net $f_\alpha\in C_c(G)$ with $\mbox{supp}(f_\alpha)\subset U$ such that $T_{f_\alpha}$ converges to $T$ in the ultrastrong topology of $\mathcal{B}(L^\Phi(G)).$
\end{lem}
\begin{lem}
Let $T,S\in PM_\Psi(G)$ be such that their supports are compact. Then $\mbox{supp}(TS)\subset\mbox{supp}(T)\mbox{supp}(S).$
\end{lem}
\begin{proof}
The proof of this lemma follows from Corollary \ref{equal}, Lemma \ref{App} and \cite[Lemma 5.2]{LK3}.
\end{proof}
Based on Lemma \ref{App} and Corollary \ref{equal}, we have the following easy corollary.
\begin{cor}\label{CSD}
The set $UCB_\Psi(\widehat{G})$ is a closed linear subspace of $PM_\Psi(G)$ and coincides with the norm closure of $\{T\in PM_\Psi(G):T\mbox{ has compact support}\}.$
\end{cor}
\begin{cor}\label{UCFAlg}
The space $UCB_\Psi(\widehat{G})$ is a closed subalgebra and $B_\Phi(G)$-submodule of $PM_\Psi(G).$ 
\end{cor}
\begin{proof}
The fact that $UCB_\Psi(\widehat{G})$ is a closed subalgebra is a consequence of the above results on the support of an element of $PM_\Psi(G).$ Further, it is clear from the definition of $UCB_\Psi(\widehat{G})$ that $UCB_\Psi(\widehat{G})$ is a $B_\Phi(G)$-submodule of $PM_\Psi(G).$
\end{proof}
\begin{cor}
\mbox{}
\begin{enumerate}[(i)]
\item $\mathscr{M}(\widehat{G})$ and $\mathscr{M}^d(\widehat{G})$ are closed subalgebras and $B_\Phi(G)$- submodules of $PM_\Psi(G).$
\item $\mathscr{M}(\widehat{G})\subset UCB_\Psi(\widehat{G})\cap WAP_\Psi(\widehat{G}).$
\item $\mathscr{M}^d(\widehat{G})\subset UCB_\Psi(\widehat{G})\cap AP_\Psi(\widehat{G}).$
\end{enumerate}
\end{cor}
\begin{proof}
(i) follows from Corollary \ref{UCFAlg}. (ii) and (iii) follows from the fact that $M(G)$ and $\ell^1(G)$ are contained in $UCB_\Psi(\widehat{G})\cap WAP_\Psi(\widehat{G})$ and $UCB_\Psi(\widehat{G})\cap AP_\Psi(\widehat{G})$ respectively.
\end{proof}
\begin{prop}\label{UCB_subset_C}
Let $G$ be a locally compact group. Then $UCB_\Psi(\widehat{G})\subset C_\Psi(\widehat{G}).$
\end{prop}
\begin{proof}
As $PF_\Psi(G)$ is a vector space, note that, in order to prove this proposition, it is enough to show that $ST$ and $TS$ are in $PF_\Psi(G)$ for any $S\in PF_\Psi(G)$ and $T\in UCB_\Psi(\widehat{G}).$ Further, as $C_c(G)$ is norm dense in $PF_\Psi(G)$ we can, without loss of generality, assume that $S=T_f$ for some $f\in C_c(G).$ Now, by Corollary \ref{CSD}, we can assume that $T$ has compact support. On the other hand, by Corollary \ref{equal} and Lemma \ref{App}, it is enough to assume that $T$ is of the form $T_g$ for some $g\in C_c(G).$ Now the conclusion follows from the fact that $T_{f*g}=T_fT_g.$
\end{proof}
\begin{cor}
The space $C_\Psi(\widehat{G})$ is a closed subalgebra and an $A_\Phi(G)$-submodules of $PM_\Psi(G).$
\end{cor}
\begin{proof}
The fact that $C_\Psi(\widehat{G})$ is a closed subalgebra follows from the definition of $C_\Psi(\widehat{G})$ while the other fact that $C_\Psi(\widehat{G})$ is an $A_\Phi(G)$-module is a direct consequence of Proposition \ref{UCB_subset_C}.
\end{proof}
\begin{cor}
Let $G$ be an amenable group. Then
$$WAP_\Psi(\widehat{G})\subset UCB_\Psi(\widehat{G})=\{A_\Phi(G).PM_\Psi(G)\}.$$
\end{cor}
\begin{proof}
Let $T\in WAP_\Psi(\widehat{G}).$ Since $G$ is amenable, by \cite[Theorem 3.1]{LK2}, it possesses a bounded approximate identity, say $\{u_\alpha\}.$ Then $u_\alpha.T\in UCB_\Psi(\widehat{G}),$ for all $\alpha.$ Note that $u_\alpha.T$ converges to $T$ in the weak*-topology of $PM_\Psi(G).$ On the other hand, using the definition of $WAP_\Psi(\widehat{G}),$ it follows that $\{u_\alpha\}$ has a subnet $\{u_{\alpha_\beta}\}$ such that $u_{\alpha_\beta}.T$ converges in the weak topology of $PM_\Psi(G).$ Let $T^\prime$ be the weak limit of $\{u_{\alpha_\beta}.T\}.$ Then, $T=T^\prime.$ As $UCB_\Psi(\widehat{G})$ is a vector space, norm topology and weak topology coincides and hence $T\in UCB_\Psi(\widehat{G}).$

The other equality is a consequence of \cite[Theorem 3.1]{LK2} and Cohen's factorization theorem.
\end{proof}
\begin{cor}\label{inc}
Let $G$ be a locally compact group. We have the following inclusions:
\begin{enumerate}[(i)]
\item If $G$ is amenable then $$PF_\Psi(\widehat{G})\subset AP_\Psi(\widehat{G})+PF_\Psi(\widehat{G})\subset WAP_\Psi(\widehat{G})\subset UCB_\Psi(\widehat{G});$$
\item If $G$ is discrete then $$PF_\Psi(\widehat{G})=\mathscr{M}^d(\widehat{G})=UCB_\Psi(\widehat{G})=C_\Psi(\widehat{G})\subset AP_\Psi(\widehat{G})\subset WAP_\Psi(\widehat{G});$$
\item If $G$ is compact then $$UCB_\Psi(\widehat{G})=C_\Psi(\widehat{G})=PM_\Psi(G).$$
\end{enumerate}
\end{cor}
\begin{proof}
(i) is a consequence of \cite[Theorem 3.1]{LK2}. (ii) is a consequence of the fact that the algebra $\ell^1(G)$ has the identity element $\delta_e,$ while (iii) follows from the fact that the algebra $A_\Phi(G)$ possesses the constant $1$ function. 
\end{proof}

{\bf Notation}. \mbox{ }
\begin{enumerate}[(i)]
\item If $E\subset G$ is closed, define $PM_\Psi^E(G)$ to be the weak-* closure of $$\{T\in PM_\Psi(G):supp(T)\subset E\}.$$
\item For any closed set $E\subset G,$ let $$A^E_\Phi(G)=\{u\in A_\Phi(G):supp(u)\subset E\}.$$
\item If $u\in A_\Phi(G),$ let $u^\perp=\{T\in PM_\Psi(G):u\cdot T=0\}.$
\item $S_{B}^{\Phi} = \big\{u\in B_{\Phi}(G):\|u\|_{B_\Phi}=u(e)=1 \big\} $
\item $S_{A}^{\Phi} =\big\{u\in A_{\Phi}(G):\|u\|_{A_{\Phi}}=u(e)=1 \big\} $
\end{enumerate}
\begin{cor}\label{Rel_Bet_EU}
Let $u\in A_\Phi(G)$ be such that $\mbox{supp}(u)\subset U,$ where $U$ is an open subset of $G.$ Then $PM_\Psi^{G\setminus U}\subseteq u^\perp.$
\end{cor}
\begin{rem}\label{Exis_Mean_Convex}
Let $u\in S_{B}^{\Phi}.$ Then, by proceeding as in \cite[Proposition 5.4]{LK1}, we can show the existence of a $m\in \overline{\{uv:v\in S_{A}^{\Phi}\}}^{w^*}$ such that $ v.m=m $ for all $ v\in S_{B}^{\Phi} $. In particular, $\overline{\{uv:v\in S_{A}^{\Phi}\}}^{w^*}\cap TIM_\Psi(\widehat{G})$ is nonempty.
\end{rem}
\begin{thm}
Let $G$ be a second countable locally compact group and if for some norm separable subspace $X\subset PM_\Psi(G)$ and some open neighbourhood $U$ of $e$ in $G$ such that $$UCB_\Psi(\widehat{G})\subset \overline{WAP_\Psi(\widehat{G})+X+PM_\Psi^{G\setminus U}(G)}^{\|\cdot\|_{PM_\Psi(G)}},$$ then $G$ is discrete.
\end{thm}
\begin{proof}
Since $G$ is locally compact, we shall, without loss of generality, assume that $U$ has compact closure. Since $G$ is second countable and the pair $(\Phi,\Psi)$ satisfy the $\Delta_2$-condition, it follows that $A_\Phi(G)$ is separable. In particular, $S_A^\Phi$ is separable. Let $\{u_n\}$ be a countable dense subset of $S_A^\Phi.$ Further, by assumption, the space $X$ is also separable. Let $\{T_n\}$ be a countable dense subset of $X.$ Let $u\in S_A^\Phi$ be such that $\mbox{supp}(u)\subset U.$ Then, by Remark \ref{Exis_Mean_Convex}, there exists $m_0\in \overline{\{uv:v\in S_{A}^{\Phi}\}}^{w^*}\cap TIM_\Psi(\widehat{G}).$ Hence, we fix this $m_0.$ 

We now prove the theorem in several steps. 

(i) Let $$A=\overline{\{uv:v\in S_{A}^{\Phi}\}}^{w^*}\cap \{m\in A_\Phi(G)^{''}:\left(t_{u_n}^{tt}-I\right)(m)=0\mbox{ and } \langle T_n,m \rangle=\langle T_n,m_0 \rangle\forall\ n\in\mathbb{N}\}.$$ Here $I$ denotes the identity map on $A_\Phi(G)^{''}$ and $t_u:A_\Phi(G)\rightarrow A_\Phi(G)$ is given by $t_u(v)=uv.$ It is clear that $A$ is convex and weak*-compact.

(ii) We claim that $A\subset TIM_\Psi(\widehat{G}).$ Let $m\in A.$ Since, $u_n.m=m$ for all $n\in \mathbb{N},$ it follows that $u.m=m$ for all $u\in S_A^\Phi.$ Let $u\in S_A^\Phi,$ $v\in S^\Phi_B$ and $T\in PM_\Psi(G).$ Then,
$$\langle T,v.m \rangle = \langle T,v.(u.m) \rangle = \langle T,(v.u).m \rangle  = \langle T,m \rangle.$$ Thus the conclusion follows from \cite[Theorem 5.6]{LK1}.

(iii) We now claim that the set $A$ is singleton. Observe that, by Corollary \ref{Rel_Bet_EU} and by our assumption, $$UCB_\Psi(\widehat{G})\subset \overline{WAP_\Psi(\widehat{G})+X+u^\perp}^{\|\cdot\|_{PM_\Psi(G)}}.$$ It is easy to see that $m_0\in A.$ Suppose that $m\in A.$ By Theorem \ref{WAP_Inv_Mean}, it is clear that $m=m_0$ on $WAP_\Psi(\widehat{G}).$ It is also clear from the definition of $A$ that $m=m_0$ on $X.$ Note that, by the definition of $v^\perp,$ $m(v^\perp)=\{0\}.$ Thus, $m=m_0$ on $\overline{WAP_\Psi(\widehat{G})+X+u^\perp}^{\|\cdot\|_{PM_\Psi(G)}}.$ In particular, $m=m_0$ on $UCB_\Psi(\widehat{G}).$ Let $T\in PM_\Psi(G)$ and $u\in S^\Phi_A.$ Then 
\begin{align*}
\langle T,m \rangle =& \langle T,u.m \rangle\ (\mbox{by }(ii))\\ =& \langle u.T,m \rangle \\ =& \langle u.T,m_0 \rangle\ (m=m_0 \mbox{ on }UCB_\Psi(\widehat{G}))\\ =& \langle T,u.m_0 \rangle\\  =& \langle T,m_0 \rangle\ (\mbox{again by }(ii))
\end{align*}
Thus $m=m_0$ on $PM_\Psi(G),$ showing that $A$ is just a singleton. 

(iv) Here we show that the group $G$ is discrete. The idea here is to use \cite[Proposition 6.5]{LK1}. By definition of compactness and the fact that $G$ is second countable, there exists a sequence $\{u_n\}\subset \{uv:v\in S_{A}^{\Phi}\}$ such that $u_n$ converges to $m_0$ in the weak*-topology. In particular, the sequence $\{u_n\}$ is a weakly Cauchy sequence in $A_\Phi^{\overline{U}}(G).$ As $A_\Phi^{\overline{U}}(G)$ is weakly sequentially complete (see \cite[Theorem 3.9]{LK1}), there exists $u_0\in A_\Phi(G)$ such that $\langle u_0,T \rangle=\langle T,m_0 \rangle.$ In particular, it follows that $A_\Phi(G)\cap TIM_\Psi(\widehat{G})\neq \emptyset.$ Now the conclusion follows from \cite[Proposition 6.5]{LK1}.
\end{proof}

\section*{Acknowledgement}
The first author would like to thank the University Grants Commission, India, for research grant.

\end{document}